\documentclass[12pt]{amsart}

\usepackage{enumerate}
\usepackage{graphicx,graphics}
\usepackage{amsfonts}
\usepackage{amssymb}
\usepackage{amsthm}
\usepackage{amsmath}
\usepackage{mathrsfs}
\usepackage{hyperref}
\input{xy}
\xyoption{all}

\newtheorem{theorem}{Theorem}[section]
\newtheorem{lemma}[theorem]{Lemma}

\newtheorem{question}{Question}

\newtheorem{remark}[theorem]{Remark}
\newtheorem{corollary}[theorem]{Corollary}

\newtheorem{proposition}[theorem]{Proposition}
\newtheorem{example}[theorem]{Example}

\DeclareMathOperator{\ext}{Ext}
\DeclareMathOperator{\extr}{Extr}

\DeclareMathOperator{\conv}{conv}

\numberwithin{equation}{section}

\begin{document}

\title{Extreme points and geometric aspects of compact convex sets in asymmetric normed spaces}

\author{Natalia Jonard-P\'erez and  Enrique A. S\'anchez-P\'erez}

%
%



\thanks{The first author has been supported by CONACYT (Mexico) under grant 204028.
The second author has been supported by the Ministerio de Econom\'{\i}a y Competitividad (Spain) under grant
 MTM2012-36740-C02-02.}

%

\maketitle\markboth{ N. JONARD-P\'EREZ AND E. A. S\'ANCHEZ P\'EREZ}
{EXTREME POINTS  OF COMPACT CONVEX SETS OF ASYMMETRIC SPACES}

\begin{abstract}
The Krein-Milman theorem states that every compact convex subset in a locally compact convex space is the closure of the convex hull of its extreme points. Inspired in this result, we investigate the existence of extreme points in compact convex subsets of asymmetric normed spaces. We focus our attention in the finite dimensional case, giving a geometric description of all compact convex subsets of a finite dimensional asymmetric normed space.
\end{abstract}

\section{Introduction}

An \textit{asymmetric normed space} is a real vector space $X$ equipped with a so called \textit{asymmetric norm} $q$. This  means that $q:X\to [0,\infty)$ is a function satisfying
\begin{enumerate}
\item $q(tx)=tq(x)$ for every $t\geq 0$ and $x\in X$,
\item $q(x+y)\leq q(x)+q(y)$ and
\item $q(x)=0=q(-x)$ if and only if $x=-x=0$.
\end{enumerate}
Any asymmetric norm induces a non symmetric topology on $X$ that is generated by the asymmetric open balls $B_q(x,\varepsilon)=\{y\in X\mid q(y-x)<\varepsilon\}.$ This topology is a $T_0$ topology in $X$ for which the vector sum on $X$ is continuous. However, in general this topology is not even Hausdorff and the scalar multiplication  is not continuous. Thus $(X,q)$ fails to be a topological vector space. 

Any asymmetric normed space $(X,q)$ has an associated (symmetric) norm. That is the norm $q^{s}:X\to [0,\infty)$ defined by the formula:
$$q^{s}(x)=\max\{q(x), q(-x)\}.$$

Compactness on these spaces has been widely studied  and nowadays there are some very interesting results that describe the general structure of  compact sets in asymmetric normed spaces (see  \cite{afgs}, \cite{new}  and \cite{gar}). When the linear space has finite (linear) dimension, it is known that the separation axiom $T_1$ implies $T_2$ ---and thus normed--- and therefore all the properties concerning the topology of those spaces are perfectly known. The general case is then when the topology has no other separation axiom  than $T_0$. 

An asymmetric normed lattice is a classic example of a non-Hausdorff asymmetric normed space,. Namely,  if $(X,\|\cdot\|, \leq)$ is a Banach lattice, there is a canonical way to define an asymmetric norm in $X$ by means of the formula 
$$q(x):=\|x\vee 0\|,\quad x\in X.$$
In this case the asymmetric norm $q$ is called  an \textit{asymmetric lattice norm} and the pair $(X,q)$ is an \textit{asymmetric normed lattice}.  This special kind of asymmetric normed spaces is important and interesting, mainly by the applications in theoretical computer science, and particularly in complexity theory 
(see, e.g. \cite{gar4}).

An interesting problem related with compactness in asymmetric normed spaces is the existence of a so called center. Namely, we say that $K'$ is a \textit{center} for a compact subset $K$ of an asymmetric normed space $(X,q)$ if $K'$ is   $q^{s}$-compact  (compact in the topology generated by the norm $q^{s}$)  and $$K'\subset K\subset K'+\theta(0)$$
where $\theta(0)=\{x\in X\mid q(x)=0\}$.
A $q$-compact set with a center is called \textit{stronly $q$-compact}  (or simply, \textit{strongly compact}).  
It is well known that not all $q$-compact sets in an asymmetric normed space are strongly compact (see \cite[Example 12]{afgs} and \cite[Example 4.6]{new}). However, in certain cases the existence of a center characterizes the $q$-compactness (\cite[Section 5]{afgs}).
Moreover, in \cite{new} it was proved that strong compactness and compactness coincide in the class of $q^{s}$-closed compact sets in a finite dimensional asymmetric normed lattice.  Furthermore,  it was proved in in \cite{Jon San}  that every $q$-compact convex set in a $2$-dimensional asymmetric normed lattice is strongly compact (even if it is not $q^{s}$-closed). This result shows that convexity plays an important role while working with compact sets in asymmetric normed spaces.

 In relation to convexity, a big effort has been made to translate the classic results of functional analysis to the non asymmetric case (see \cite{cobzas}).   One of these results is the Krein-Milman Theorem.
It states that every compact convex subset of a locally convex space is the closure of the convex hull of its extreme points. In particular, each compact convex subset of a locally convex space has at least an extreme point.
Concerning the asymmetric case, it was proved in \cite{cobzas} that every $T_1$ asymmetric normed space satisfies the Krein-Milman Theorem. 

However, in the general case this result is not longer true, not even in finite dimensional asymmetric normed spaces.  For example,
let us consider the asymmetric norm $|\cdot |_a:\mathbb R\to [0,\infty)$ in $\mathbb R$ given by $|t|_a=\max\{0,t\}$. The set $(-1,1]$ is a compact convex set and its only extreme point is $1$. Thus $(-1,1]$ cannot be the convex hull of its extreme points. Even more, the closure of $\{1\}$ (in the asymmetric topology of $\mathbb R_a$) coincides with the interval $[1,\infty)$, which is far from being $(-1,1]$.

The main purpose of this work is to study the geometric structure of compact convex sets in asymmetric normed spaces. In particular, we are interested in the existence of extreme points in compact convex sets of these spaces. Although  we focus our attention on the finite dimensional case, some of our results are also valid in the infinite dimensional setting.

We prove in Theorem~\ref{t:main} that every $q$-compact convex set of a finite dimensional asymmetric normed space has at least one extreme point. Later, in section~\ref{seccion geometric},  extreme points of convex compact sets are used to give a geometric description of the compact convex subsets of a finite dimensional asymmetric normed space.  Somehow, this geometric description is related with the notion of a strongly compact set, and it generalizes in some aspects a previous work due by the autors in the 2-dimensional case (see \cite{Jon San}). 

\section{Preliminaries}

In this section we recall some important definitions and results that will be used throughout the paper. 

Consider a convex set $A$ contained in a linear space $X$. A point 
$x\in A$ is an \textit{extreme} point of $A$ if $x=y=z$ whenever $y,z\in A$ and $x=\lambda y+(1-\lambda)z$ for some $\lambda\in (0,1)$. Similarly, an open half line $R=\{a+tb\mid a,b\in X, ~t> 0\}$ is called an \textit{extreme ray} of $A$ if $y,z\in R$ whenever $\lambda y+(1-\lambda)z\in R$, where  $y, z\in A$ and $\lambda\in (0,1)$. If $R=\{a+tb\mid a,b\in X, ~ t> 0\}$ is an extreme ray of $A$ such that the extreme $a$ lies in $A$, then $a$ is an extreme point of $A$. Further, if $A$ is closed then the extreme of every extreme ray is contained in $A$.

In certain cases, the set of extreme points and the set of extreme rays of a convex set $A$ determine  the structure  of $A$ itself. Indeed, in 1957,  V. L. Klee gave a generalization of the Krein-Milman Theorem for locally compact closed subsets of a locally convex linear space. This result will be used in this work and we state it as follows.

\begin{theorem}[\cite{Klee}]\label{t:klee}
Let $C$ be a locally compact closed convex subset of a locally convex (Hausdorff) linear space. Then 
$$C=\overline{\conv \big(\ext (C)\cup\extr (C)\big)}.$$
Additionally, if $C$ is finite dimensional then actually $$C=\conv \big(\ext (C)\cup\extr (C)\big).$$
\end{theorem}

Recall that $\conv (A)$ denotes the convex hull of $A$, i.e., the smallest convex set containing $A$.

Let $(X,q)$ be an asymmetric normed space and $q^{s}:X\to [0,\infty)$ the associated norm given by
$$q^{s}(x)=\max\{q(x),q(-x)\}.$$

 For every $x\in X$ and $\varepsilon\geq 0$, we will use the following notation:
$$B_q(x, \varepsilon) =\{y\in X\mid q(y-x)<\varepsilon\}$$
$$B_q[x,\varepsilon]=\{y\in X\mid q(y-x)\leq \varepsilon\}.$$
Similarly  we define $B_{q^{s}}(x,\varepsilon)$ and $B_{q^{s}}[x,\varepsilon]$.
Observe that 
\begin{equation}\label{e:bolas abiertas trasladadas}
B_q(x,\varepsilon)=x+B_q(0,\varepsilon),
\end{equation}
\begin{equation}\label{e:bolas cerradas trasladadas}
B_q[x,\varepsilon]=x+B_q[0,\varepsilon].
\end{equation}
The sets $B_q(x,\varepsilon)$ are called \textit{open balls}, while $B_q[x,\varepsilon]$ are called \textit{closed balls}. These sets are always convex. Further, if $\varepsilon>0$, the set $B_q(0,\varepsilon)$ is \textit{absorbing}, i.e., for every $x\in X$ there exists $t>0$ such that $tx\in B_q(0,\varepsilon)$ (see \cite{cobzas}).

The topology in $(X,q)$ is the one generated by the sets $B_q(x,\varepsilon)$ with $x\in X$ and $\varepsilon>0$. This topology is usually denoted by $\tau_q$. In certain cases, we will need to use the topology generated by the norm $q^{s}$, i.e., the topology $\tau_{q^{s}}$ determined by the sets $B_{q^{s}}(x,\varepsilon)$ with $x\in X$ and $\varepsilon >0$. For this reason, in order to avoid any confusion, it is important to distinguish both topologies  at the moment of dealing with them.  For instance,
we will say that a set $A\subset X$ is $q$-compact ($q^{s}$-compact) if it is compact in the topology $\tau_q$ ($\tau_{q^{s}}$). Similarly, we define the notion of $q$-open, $q$-closed, $q$-continuous ($q^{s}$-open, $q^{s}$-closed, $q^{s}$-continuous), etc.
Furthermore, given a set $A\subset X$, we will use the symbol $\overline{A}$ and $\overline{A}^{s}$ to denote the closure of $A$ with respect to the topology $\tau_q$ and $\tau_{q^{s}}$, respectively.  In general, the closed balls $B[x,\varepsilon]$ are not $q$-closed. However, as we will prove in Lemma~\ref{l:q continua} below, these sets are always $q^{s}$-closed.

Given an asymmetric normed space $(X,q)$, the set $\theta(x)$ is defined as $$\theta (x)=\{y\in X\mid q(y-x)=0\}.$$
The set $\theta (0)$ consisting of all points  $y$ with $q(y)=0$ is a convex cone  of particular importance in the study of compactness in asymmetric normed spaces (see \cite{gar}). 

In this note we will use several times the following useful properties involving the set $\theta (0)$.

\begin{proposition}\label{p:theta} Let $(X,q)$ be an asymmetric normed space.  
\begin{enumerate}[\rm(1)]
\item For any $q$-open subset $U\subset X$, $U=U+\theta(0)$.
\item A set $K\subset X$ is $q$-compact if and only if $K+\theta (0)$ is $q$-compact.
\end{enumerate}
\end{proposition}
For the  proof of this result, the reader can consult Lemma 4 and Proposition 5 of \cite{gar}.


\section{Geometric structure of compact convex subsets}

The purpose of this section consists on proving basic lemmas concerning the topology and geometry of asymmetric normed spaces  that will be used in the following sections of this paper. Some of these results may be well-known, but we include them  here for the sake of completeness.
 
 \begin{lemma}\label{l:q continua}
 Let $(X,q)$ be an asymmetric normed space. The map $q:(X,q^{s})\to \mathbb R$  is continuous ($q^{s}$-continuous). In particular, the $q$-open balls $B_q(z,\varepsilon)$ are $q^{s}$-open and the sets $B_q[z,\varepsilon]$ are $q^{s}$-closed.
 \end{lemma}
\begin{proof}
For any  $x,y \in X$ we have that 
$$q(x)=q(x-y+y)\leq q(x-y)+q(y)\leq q^{s}(x-y)+q(y)$$
and therefore $q(x)-q(y)\leq q^{s}(x-y)$. Symmetrically, 
$$q(y)=q(y-x+x)\leq q(y-x)+q(x)\leq q^{s}(y-x)+q(x)=q^{s}(x-y)+q(x).$$
The previous inequality  implies that $q(y)-q(x)\leq q^{s}(x-y)$ and thus $|q(x)-q(y)|\leq q^{s}(x-y)$. This proves that $q$ is $q^{s}$-continuous.

Finally, for every $\varepsilon>0$, the set $B_q(0,\varepsilon)$ is the inverse image of the open set $(-\infty, \varepsilon)$ and therefore it is open in $(X, q^{s}).$ Similarly, $B_q[0,\varepsilon]$ is the inverse image of the closed interval $[0,\varepsilon]$ and so it is closed in $(X, q^{s})$. To complete the proof we simply observe that $B_q(z,\varepsilon)$ ($B_q[z,\varepsilon]$) is the translation of a $q^{s}$-open ($q^{s}$-closed) set and therefore it is $q^{s}$-open ($q^{s}$-closed). Now the proof is complete.
\end{proof}

\begin{lemma}\label{l:qsclosure is compact}
Let $K$ be a $q$-compact set. Then the $q^{s}$-closure $\overline{K}^{s}$ is $q$-compact too.
\end{lemma}

\begin{proof}
Let $\mathcal U$ be a $q$-open cover for $\overline{K}^{s}$. Thus, for every $x\in \overline{K}^{s}$ there exist $\delta_x>0$  and $U_x\in\mathcal U$ such that $B_q(x,2\delta_x)\subset U_x$.
Obviously the family $\{B_q(x,\delta_x)\}_{x\in \overline{K}^{s}}$ is a $q$-open cover for $K$ and therefore we can pick a finite subcover $\{B_{q}(x_i,\delta_{x_i})\}_{i=1}^{n}$.
Then, we have 
$$K\subset\bigcup\limits_{i=1}^{n} B_q(x_i, \delta_{x_i})\subset \bigcup\limits_{i=1}^{n} B_q[x_i,\delta_{x_i}].$$
Since $\bigcup_{i=1}^{n} B_q[x_i,\delta_{x_i}]$ is closed, we also have that
$$\overline{K}^{s}\subset \bigcup\limits_{i=1}^{n} B_q[x_i,\delta_{x_i}]\subset \bigcup\limits_{i=1}^{n} B_q(x_i,2\delta_{x_i})\subset \bigcup\limits_{i=1}^{n}U_{x_i}$$
This proves that $\{U_{x_i}\}_{i=1}^{n}$ is a finite subcover for $\overline{K}^{s}$ and therefore it is $q$-compact,  as desired.
 
\end{proof}

\begin{lemma}\label{l:no line in ball}
For every asymmetric normed space $(X,q)$ and for every $\varepsilon>0$, the $q$-open ball $B_q(0,\varepsilon)$ contains no line.
\end{lemma}

\begin{proof}
First we will prove that $B_q(0,\varepsilon)$ contains no line trough the origin. Suppose that $L=\{tx_0\mid t\in\mathbb R\}$ is contained in $B_q(0,\varepsilon)$ for some fixed point $x_0\in X\setminus\{0\}$. If $q(tx_0)=r>0$ for some $t$,  then $\frac{2\varepsilon}{r}(tx_0)\in L\subset B_q(0,\varepsilon)$. We thus have 
$$q\bigg(\frac{2\varepsilon}{r}tx_0\bigg)=\frac{2\varepsilon}{r}q(tx_0)=\frac{2\varepsilon}{r}r>\varepsilon,$$
which is a contradiction. Therefore we can conclude that $q(u)=0$ for every $u\in L$. However this is a contradiction too since $q(u)=0=q(-u)$ implies that $u=0$ for every $u\in L$.

Now we will prove that $B_q(0,\varepsilon)$ does not contain any kind of line.  Suppose that $\Lambda=\{z+ty\mid t\in\mathbb R\}$ is contained in $B_q(0,\varepsilon)$ for some fixed points $z\in X$ and $y\in X\setminus\{0\}$.

Consider the element $-z$. Since the set $B_q(0,\varepsilon)$ is absorbing (see \cite{cobzas}), there exists some $r_0>0$ with the property that $r_0(-z)\in B_q(0,\varepsilon)$; thus, there is $r <0 $ such that $r z \in B_q(0,\varepsilon)$.
Now, using the fact that $B_q(0,\varepsilon)$ is convex, we get that
$$s(rz)+(1-s)(z+ty)\in B_q(0,\varepsilon)\quad\text{for every }s\in[0,1],\text{ and }t\in\mathbb R.$$
In particular, taking $s=\frac{1}{1-r}\in (0,1)$ we get that
$$\frac{r}{1-r}z+\big(1-\frac{1}{1-r}\big)(z+ty)=\big(1-\frac{1}{1-r}\big)ty\in B_q(0,\varepsilon)$$
for every $t\in\mathbb R$. Since $1-1/(1-r)\neq 0$, we conclude that $ky\in B_q(0,\varepsilon)$ for every $k\in\mathbb R$, which contradicts the  first part of this proof. Now the lemma is proved.

\end{proof}

Since for every $\varepsilon >0$  we have $B_q[0,\varepsilon]\subset B_q(0,2\varepsilon)$, no closed ball $B_q[0,\varepsilon]$ contains a line either. Further, since any translation is an affine bijection, we infer  from equalities (\ref{e:bolas abiertas trasladadas}) and (\ref{e:bolas cerradas trasladadas}) the following corollary.

\begin{corollary}\label{c:no line in ball}
No ball (open or closed) in an asymmetric normed space contains a line.
\end{corollary}

\begin{corollary}\label{c:noline in compact}
If $K$ is a $q$-compact subset of an asymmetric normed space $(X,q)$, then $K$  contains no line.
\end{corollary}

\begin{proof}
Let $\{B_q(0,m)\}_{m\in\mathbb N}$ be the family of all open balls centered in the origin and having radius $m\in\mathbb N$. Clearly   
$\{B_q(0,m)\}_{m\in\mathbb N}$
is an open cover for $K$, and thus we can find $M\in \mathbb N$ such that $K\subset B_{q}(0,M)$. By corollary~\ref{c:no line in ball}, there is no line  completely contained in $B_q(0,M)$ and therefore $K$ contains no line either.
\end{proof}

\begin{lemma}\label{l:muchas propiedades}
For any asymmetric normed space $(X,q)$, the following statements hold.

\begin{enumerate}[\rm(1)]
\item $\theta (0)$ is $q^{s}$-closed. 
\item If the closed ray $R=\{tx\mid t\geq 0\}$ is contained in the ball $B_{q}(0,\varepsilon)$ then $R\subset \theta (0)$.
\item If the ray $R=\{tx+y\mid t\geq 0\}$ is contained in the ball $B_{q}(0,\varepsilon)$ then $R'=\{tx\mid t\geq 0\}\subset \theta (0)$.
\item If the ray $R=\{tx+y\mid t\geq 0\}$ is contained in the ball $B_{q}(z,\varepsilon)$ then $R'=\{tx\mid t\geq 0\}\subset \theta (0)$. 
\item If the ray $R=\{tx+y\mid t\geq 0\}$ is contained in a $q$-compact set $K\subset R$, then  $R'=\{tx\mid t\geq 0\}\subset \theta (0)$. 

\end{enumerate}

\end{lemma}

\begin{proof}
(1) This follows immediately from Lemma~\ref{l:q continua}.

(2) Suppose that $R\cap \big(B_{q}(0,\varepsilon)\setminus \theta (0)\big)\neq \emptyset$. So, there exists $t_0>0$ such that $0<q(t_0x)<\varepsilon$. In this case, for every $M>\varepsilon/q(t_0x)$ we get that $(Mt_0)x\in R\subset B_{q}(0,\varepsilon)$ but
$$q(Mt_0x)=Mq(t_0x)>\varepsilon.$$
From this contradiction we conclude that $R\subset \theta (0)$.  

(3) First observe that $y=0x+y\in B_q(0,\varepsilon)$. Now, since $B_q(0,\varepsilon) $ is  absorbing, it is possible to find $\delta <0$ with the property that $\delta y\in B_{q}(0,\varepsilon)$.

Since $B_q(0,\varepsilon)$ is convex, we have that
$$s(\delta y)+(1-s)(tx+y)\in B_{q}(0,\varepsilon)\quad\text{for every }s\in[0,1],\text{ and }t\geq 0.$$
In particular, if we take $s=\frac{1}{1-\delta}\in (0,1)$ we get that
$$\frac{\delta}{1-\delta}y+\big(1-\frac{1}{1-\delta}\big)(tx+y)=\big(1-\frac{1}{1-\delta}\big)tx\in B_q(0,\varepsilon)$$
for every $t\geq 0$. Since $1-1/(1-\delta)> 0$, we conclude that $kx\in B_q(0,\varepsilon)$ for every $k\geq 0$.  Now (2) implies that $R'\subset B_q(0,\varepsilon)$, as desired.

(4) Let $t=q(z)$. If $y\in B_q(z,\varepsilon)$, then 
$$q(y)=q(y-z+z)\leq q(y-z)+q(z)<\varepsilon+t.$$
Thus, $R\subset B_q(z,\varepsilon)\subset B_q(0,\varepsilon+t)$. Now the result follows directly from (3).

(5) The family $\{B_q(0,n)\}_{n\in\mathbb N}$ is an open cover for the compact set $K$ and therefore we can find $n_0\in \mathbb N$ such that $R\subset K\subset B_q(0,n_0)$. Now the result follows directly from (3).

\end{proof}

For every $q$-compact set $K$, the associated  $q$-compact set $K+\theta(0)$ will play an important role. A useful property of this set is the following.

\begin{lemma}\label{l:K mas theta cerrado}
Let $K$ be a $q$-compact set in an asymmetric normed space $(X,q)$. Then the set $K+\theta (0)$ is $q^{s}$-closed.
\end{lemma}

\begin{proof} It is enough to prove that $\overline{K+\theta(0)}^{s}\subset K+\theta(0)$. In order to do this, let us consider a point $x\in \overline{K+\theta(0)}^{s}$ and suppose that $x$ is not in $K+\theta(0)$. This implies that for every $a\in K$, the point $x-a$ is not in $\theta (0)$ or equivalently, $q(x-a)>0$ for every $a\in K$.
This means that $\mathcal U=\{B_q(a,\delta_a)\}_{a\in K}$ where $\delta_a=q(x-a)/2$ is an open cover for $K$ and according to Proposition~\ref{p:theta}, it also covers $K+\theta(0)$. Since $K$ is compact we can extract a finite subcover $\{B_q(a_i, \delta_{a_i})\}_{i=1}^{n}$. 
Thus we have
$$K+\theta(0)\subset \bigcup_{i=1}^{n}B_q(a_i,\delta_{a_i})+\theta(0)= \bigcup_{i=1}^{n}B_q(a_i,\delta_{a_i})\subset\bigcup_{i=1}^{n}B_q[a_i,\delta_{a_1}].$$
Since the union $\bigcup_{i=1}^{n}B_q[a_i,\delta_{a_1}]$ is $q^{s}$-closed, we also have that 
$$x\in \overline{K+\theta (0)}^{s}\subset \bigcup_{i=1}^{n}B_q[a_i,\delta_{a_1}],$$
which is a contradiction since $q(x-a_i)=2\delta_{a_i}>\delta_{a_i}$ for every $i=1,\dots, n$.
Therefore we can conclude that $\overline{K+\theta(0)}^{s}\subset K+\theta(0)$, as desired.
\end{proof}

\section{Extreme points in compact convex subsets}

\begin{theorem}\label{t:extrme points}
Let $(X,q)$ be an asymmetric normed space. Suppose that $K\subset X$ is a $q$-compact convex subset of $X$.  Then the set of extreme points of $K+\theta (0)$  is contained in $K$. 
\end{theorem}

\begin{proof}
Suppose that $x_0\in K+\theta(0)$ is an extreme point of $K+\theta(0)$ and assume that $x_0$ is not in $K$. 
We claim that for every $z\in K$, there exists $\varepsilon_z>0$ such that $x_0\notin B_q(z,\varepsilon_z)$. Indeed, let $z\in K$. If $x_0$ lies  in  $z+\theta(0)$ we can find $y\in \theta(0)\setminus\{0\}$ such that  $x_0=z+y$. Obviously $2y\in \theta(0)$ and therefore $z+2y\in K+\theta(0)$. Now we can write $x_0$ as the convex combination 
$$x_0=\frac{1}{2}(z+2y)+\frac{1}{2}z$$
 where both $(z+2y)$ and $z$ are distinct elements of $K+\theta(0)$. This last expression of $x_0$ is a contradiction since $x_0$ is an extreme point of $K+\theta (0)$. Thus, $x_0$ cannot be in $z+\theta(0)$ for any $z\in K$ and therefore 
 $$\varepsilon_z:=q(x_0-z)/2>0.$$
From this last inequality we conclude that $x_0\notin B_q(z,\varepsilon_z)$ for every $z\in K$. This proves the claim.

Now, the family $\{B_q(z,\varepsilon_z)\}_{z\in K}$ is a cover for $K$  and

$$K\subset \bigcup_{z\in K} B_q(z,\varepsilon_z)\subset  X\setminus\{x_0\}. $$

Moreover, by Proposition~\ref{p:theta}(1), we have that
$$K+\theta (0)\subset \bigg(\bigcup_{z\in K}B_q(z,\varepsilon_z)\bigg)+\theta (0)=\bigcup_{z\in K}B_q(z,\varepsilon_z)\subset X\setminus\{x_0\},$$
 which contradicts the fact that $x_0\in K+\theta (0)$.  Now the proof is complete.
\end{proof}

Let $K$ be a $q$-compact convex set in an asymmetric normed space $(X,q)$. By Lemma~\ref{l:qsclosure is compact}, the set $K+\theta (0)$ is $q^{s}$-closed. Further, since $K+\theta (0)$ is $q$-compact (Proposition~\ref{p:theta}(2)), it follows from Corollary~\ref{c:noline in compact}, that $K+\theta (0)$ contains no line.
If, additionally, the set $K+\theta (0)$ is locally compact in the topology determined by $q^{s}$ (for example, if $K+\theta (0)$ is finite dimensional), it follows from Theorem~\ref{t:klee}  that $K+\theta (0)$ is the closed convex hull of its extreme points and extreme rays. In particular, $K+\theta (0)$ has at least one extreme point. This fact, in combination with Theorem~\ref{t:extrme points} gives the following result.

\begin{theorem}\label{t:main}
Let $K$ be a $q$-compact convex subset of an asymmetric normed space $(X,q)$ 
with the property that $K+\theta (0)$ is $q^{s}$-locally compact. Then $K$ has at least one extreme point. In particular, if $K+\theta (0)$  has finite dimension, then $K$ has at least one extreme point.
\end{theorem}

\begin{corollary}
Every compact convex subset $K$ of  a finite dimensional asymmetric normed space $(X,q)$ has at least one extreme point.
\end{corollary}

In contrast with the normed case, let us observe that Theorem~\ref{t:main} is the best we can say about extreme points in $q$-compact convex sets. For instance, in any asymmetric normed space $(X,q)$,  the set $\theta (x)=x+\theta (0)$ is a $q$-compact convex set for whom its only extreme point is $x$ itself.

The main application that we have in mind is the case of the finite dimensional asymmetric spaces. However, as we can see in the following example, there are other cases in which  Theorem~\ref{t:main} is useful too.

\begin{example}
 Consider a lattice norm $\| \cdot\|_n$ in the finite dimensional space $\mathbb R^n$ and take the asymmetric norm $q_n(\cdot):=\| \cdot \, \vee 0\|_n$. 
Let $(X,\| \cdot\|)$ be a Banach space. We define the new space
$
(\mathbb R^n \times X, q ),
$
where $q(\cdot)= q_n(\cdot)+ \| \cdot\|$.
Take  a  compact convex closed set $Z$   in $X$ and a $q_n$-compact convex set $A$ in $\mathbb R^n$. Then the set $B$ defined as the product $A \times Z$ is a $q$-compact set satisfying that  $B+ \theta(0)$ is locally compact ---although it may not have finite dimension---. This implies that it has an extreme point, as an application of Theorem \ref{t:main}. Clearly, B is not $q^s$-compact if $A$ is not $q_n^{s}$ compact.
\end{example}

\begin{remark}
In the case of asymmetric norms on Riesz spaces that are defined by means of lattice norms as $\| \cdot \vee0\|$, it must be noted that in general, $q$-compact sets do not have a supremum belonging to the set. Proposition 4.2 in \cite{new} shows that sets of asymmetric lattices having a supremum belonging to the set are compact, but this is not a general way for finding extreme points of $q$-compact convex sets in asymmetric lattices. The natural extreme point (the supremum) does not belong to the set in general, so Theorem \ref{t:main} is necessary to assure the existence of extreme points. Let us show an easy example. Take the set 
$$A:= \overline{\conv(\{e_n/n \mid n \in \mathbb N\})}^s\subset \ell^1.$$
 The lattice supremum that can be found in $\ell^\infty$ is $\sum_{n=1}^\infty e_n/n$. It does not belong to $\ell^1$, so the set $A$ does not have a supremum. However, each element $e_n/n$ is an extreme point of $A+\theta(0)$, that is, a compact convex set in $(\ell^1, \| \cdot \vee 
  0\|_{\ell^1})$. Clearly, it is not $q^s$-compact.
\end{remark}

\section{Geometric structure of finite dimensional compact convex subsets}\label{seccion geometric}

%
%
%
%
%
%
%

For any $q$-compact convex subset $K$ in an asymmetric normed space $X$, let us denote by $E(K)$ the extreme points of $K+\theta (0)$, and $S(K)$ the convex hull of $E(K)$.
By Theorem \ref{t:extrme points}, if $K+\theta(0)$ is $q^s$-locally compact, the set $E(K)$ is non-empty and therefore $S(K)$ is non empty either.
 In the following theorem, we will show that set $S(K)$ plays the role of a center. However, the set $S(K)$ may fail to be
$q^{s}$-compact. 

\begin{theorem}\label{t:structure compact convex sets}
Let $(X,q)$ be an asymmetric normed space and $K$ a $q$-compact convex subset of $X$ such that  $K+\theta (0)$ has finite dimension (for example, if $X$ is finite dimensional). Then 
$$S(K)\subset K\subset S(K)+\theta (0)=K+\theta(0).$$

\end{theorem}

\begin{proof}
First observe that $E(K)=\ext (K+\theta (0))\subset K$ (Theorem~\ref{t:extrme points}) and due to the convexity of the set $K$, it follows that $S(K)\subset K$. This proves the left contention.

Now, to prove the right contention it is enough to show that $K+\theta (0)$ is contained in $S(K)+\theta (0)$ (observe that this also proves that $K+\theta(0)=S(K)+\theta(0)$). 
To see this, first observe that since $K+\theta (0)$ is $q$-compact, it contains no line (Corollary~\ref{c:noline in compact}) and hence, due to Theorem~\ref{t:klee} it is the convex hull of its extreme points and extreme halflines. Namely, 
every point  $x\in K+\theta (0)$ can be written as a convex combination of  extreme points and points lying in the extreme half lines of $K+\theta (0)$. To be more precise, there exist points $x_1,\dots, x_n\in E(K)$, $y_1,\dots, y_{m}\in \extr (K+\theta (0))$ and scalars $\lambda_1,\dots, \lambda_n$  and $\gamma_1,\dots, \gamma_m$ such that:
$$x=\sum_{i=1}^{n}\lambda_ix_i+\sum_{j=1}^{m}\gamma_jy_j, \quad \quad \sum_{i=1}^{n}\lambda_{i}+\sum_{j=1}^{m}\gamma_{j}=1,$$
being $\lambda_i$ and $\gamma_j$ real positive numbers.
 Now, using Lemma~\ref{l:muchas propiedades} and the fact that $K+\theta(0)$ is $q^{s}$-closed, we have that  every extreme ray of $K+\theta (0)$ must be of the form 
 $$R=\{tz+u\mid z\in \theta(0),~t\geq 0\}$$
 with $u$  an extreme point of $K+\theta (0)$. Thus, for each $j=1,\dots ,m$, there exist $z_j\in \theta (0)$, $t_j\geq 0$ and $u_j\in E(K)$ such that
 $$y_j=t_jz_j+u_j.$$
 Now, we can express the point $x$ as 
 \begin{align*}
 x&=\sum_{i=1}^{n}\lambda_ix_i+\sum_{j=1}^{m}\gamma_jy_j=
 \sum_{i=1}^{n}\lambda_ix_i+\sum_{j=1}^{m}\gamma_j(u_j+t_jz_j)\\
 &=\sum_{i=1}^{n}\lambda_ix_i+\sum_{j=1}^{m}\gamma_ju_j+\sum_{j=1}^{m}\gamma_jt_jz_j.
 \end{align*}
 Finally, observe that $\sum_{i=1}^{n}\lambda_ix_i+\sum_{j=1}^{m}\gamma_ju_j$ is a convex combination of extreme points of $K+\theta(0)$ and thus it lies in $S(K)$. While 
 $$\sum_{j=1}^{m}\gamma_jt_jz_j=\sum_{j=1}^{m}\gamma_jt_jz_j+(1-\sum_{j=1}^{m}\gamma_j)0$$
 is a convex combination of elements of $\theta (0)$ and due to the convexity of this last subset we can conclude that $\sum_{j=1}^{m}\gamma_jt_jz_j$ belongs to $\theta (0)$. Therefore $x\in S(K)+\theta(0)$, and now the theorem is proved.
\end{proof}
 
 \begin{corollary}
Let $K$ be a $q$-compact convex subset in an asymmetric normed space $(X,q)$ such that $K+\theta (0)$ has finite dimension.  
If $K'\subset X$ is any subset satisfying 
$$S(K)\subset K'\subset S(K)+\theta (0)$$
then $K'$ is $q$-compact. 
\end{corollary}

\begin{proof}
Consider $\mathcal U$ a $q$-open cover of $K'$.  Since $S(K)\subset K'$, $\mathcal U$ is a $q$-open cover of $S(K)$ too. Then, by Proposition~\ref{p:theta}(1) and Theorem \ref{t:structure compact convex sets}, we get
 $$K\subset S(K)+\theta (0)\subset \bigcup\limits_{U\in\mathcal U}U+\theta (0)=\bigcup\limits_{U\in\mathcal U}U.$$
 From this chain of contentions, we conclude that $\mathcal U$ is a $q$-open cover for the $q$-compact set $K$ and thus we can extract a finite subcover $U_1,\dots ,U_p$. Now, since $S(K)$ is contained in  $K$,  the family  $U_1,\dots ,U_p$ is a finite cover for $S(K)$ too. Additionally, since  $\bigcup_{i=1}^{p}U_i$ is $q$-open, we infer from Proposition~\ref{p:theta}(1) that
 $$K'\subset S(K)+\theta(0)\subset \bigcup_{i=1}^{p}U_i+\theta (0)=\bigcup_{i=1}^{p}U_i,$$
 which implies that $K'$ is a $q$-compact set, as desired.
\end{proof}

In \cite{Jon San}, it was proved that every $q$-compact convex set in a $2$-dimensional asymmetric normed lattice is strongly $q$-compact. For proving that, the authors constructed a $q^{s}$-compact set $K$ (denoted by $R(K)$ in \cite{Jon San}) that coincides with the set $S(K)$ previously defined. However, if the dimension of the space is equal or greater than $3$, compact convex sets may not be strongly $q$-compact. An example of this situation is  showed in the following example, where we exhibit a $q$-compact set $K$ in a three dimensional asymmetric lattice, such that  $S(K)$ is not  $q^{s}$-compact and $K$ is not strongly $q$-compact.

\begin{example}
Let $(\mathbb R^{3}, q)$ where $q:\mathbb R^{3}\to[0,\infty)$ is the asymmetric lattice norm defined by the rule:
$$q(x)=\max\{\max\{x_i, 0\}\mid i=1,2,3\}\quad x=(x_1,x_2,x_3)\in \mathbb R^{3}.$$
Let $K=\conv(A\cup\{(0,0,0),(0,1,1)\})$ where $A$ is the set defined as 
$$A=\{(x_1,0,x_3)\mid x_1^{2}+x^{2}_3=1,\,x_1\in (0,1],\,x_3\geq 0 \}.$$
For any $q$-open cover $\mathcal U$ of $K$, there exist an element $U\in \mathcal U$ such that $(0,1,1)\in U$. This implies that 
$$(0,1,1)+\theta (0)\subset U+\theta (0)=U$$
and therefore the cover $\mathcal U$ is a $q$-open (and $q^{s}$-open) cover  for $\overline {K}^{s}$ which  is $q^{s}$-compact. Thus, we can extract a finite subcover $\mathcal{V}\subset\mathcal U$ for   $\overline {K}^{s}$. This cover $\mathcal V$ is a finite subcover for $K$ too, and then we can conclude that $K$ is $q$-compact.

In this case, $S(K)=K$ which is not $q^{s}$-compact. To finish this example, let us note that $K$ is not strongly $q$-compact. For this, simply observe that any set $K_0$ satisfying $K_0\subset K\subset K_0+\theta (0)$, must contain the set $A$. If $K_0$ is additionally $q^{s}$-compact, then it is $q^{s}$-closed too and then
$$\overline{A}^{s}\subset \overline{K_0}^{s}=K_0\subset K$$
which is impossible. 
\end{example}

\textbf{Final Remark.}
Since we use Klee's theorem, the hypothesis of $K+\theta (0)$ being locally compact in Theorem~\ref{t:main} is essential. Further, there  are many examples of closed bounded convex subsets in Banach spaces without extreme points (for example, the unitary closed ball of $c_0$). However, we don't know if there is an example of an infinite dimensional $q$-compact convex set without extreme points. So, we finish our paper with the following question.

\begin{question}
Does every $q$-compact convex set in an asymmetric normed space contain an extreme point?
\end{question}


\vspace{2cm}

\noindent[Natalia Jonard-P\'erez]  Departamento de Matem\'aticas,
Facultad de Matem\'aticas,
Campus Espinardo,
30100 Murcia, Spain, e-mail: natalia.jonard@um.es

\vspace{1cm}

\medskip

\noindent[Enrique A. S\'anchez P\'erez] Instituto Universitario de Matem\'{a}tica Pura y Aplicada, Universidad Polit\'ecnica de Valencia, Camino de Vera s/n, 46022 Valencia, Spain, e-mail: easancpe@mat.upv.es

\end{document}